\documentclass[a4paper]{amsart}

\long\def\adraft#1{{\color{black}#1}}

\usepackage[v2]{xypic}
\usepackage{amssymb}

\def\dd{{\partial}}

\usepackage{color}
\newtheorem{theorem}{Theorem} 
\newtheorem{proposition}[theorem]{Proposition}

\newtheorem{lemma}[theorem]{Lemma}
\newtheorem{definition}[theorem]{Definition}
\newtheorem{remark}[theorem]{Remark}

\def\invertb#1{\xi_{#1}}
\def\invert#1{\invertb{#1}}

\def\ccobar{\hat\Omega C}
\def\Om{\Omega}

\def\tauinv#1{\tau(#1)^{-1}}
\def\Sz{{\textrm{\rm Sz}}} 
\def\sgn{\mathop{\mathrm{sgn}}}
\def\Shuff{\mathop{\mathrm{Shuff}}}

\def\finalom#1#2#3{\omega_{#1}(#2,#3)}
\def\finalombar#1#2#3{\overline\omega_{#1}(#2,#3)}

\title{The loop group and the cobar construction}
\author{Kathryn Hess \and Andrew Tonks}

\address{Institut de g\'eom\'etrie, alg\`ebre et topologie (IGAT) \\
    \'Ecole Polytechnique F\'ed\'erale de Lausanne \\
    CH-1015 Lausanne \\
    Switzerland}
\email{kathryn.hess@epfl.ch}
\address{STORM \\
    London Metropolitan University \\
    166--220 Holloway Road \\
    London N7 8DB\\
    United Kingdom}
\email{a.tonks@londonmet.ac.uk}
\date{\today}
\begin{document}
\maketitle
\begin{abstract} We prove that for any $1$-reduced simplicial set $X$, Adams' cobar construction $\Omega CX$ on the
normalised chain complex of $X$ is  naturally a strong deformation retract of the normalised chains $CGX$ on the Kan loop group $GX$. In order to prove this result,  we extend the definition of the cobar construction
and actually obtain the existence of such a strong deformation retract
for all $0$-reduced simplicial sets.  \end{abstract}

\section*{Introduction}

There are two classical differential graded algebra models for the 
loop space on a $1$-reduced simplicial set $X$: 
Adams' cobar construction $\Omega CX$ on the
normalised chain complex{~\cite{adams}}, 
and the normalised chains $CGX$ on the Kan loop group
$GX${~\cite{kan}}. 
{Both of these models are (weakly) equivalent to $C\Omega|X|$, 
the chains on the loop space of the realisation $|X|$}. 

In this article we show that $\Omega CX$ is actually a strong deformation retract of $CGX$, opening up the possibility of applying the tools of homological algebra to transfering perturbations of algebraic structure from the latter to the former.

\medskip\par\noindent{\bf Theorem.}
{\em For any $1$-reduced 
simplicial set $X$ 
there is a 
natural strong deformation retract of chain complexes
\begin{equation}\label{eqn:sdr}
\xymatrix{\Omega CX\rto<0.8ex>^-\phi&CGX\!\!\lto<0.8ex>^-\psi 
\;\;\;.\ar`r^d^{\qquad\qquad \Phi}[]}
\end{equation}
{Here} $\psi\phi$ is the identity map on $\Omega CX$
{and} $\Phi$ is a chain homotopy from $\phi\psi$ to the identity
map on $CGX$. {Furthermore} both $\phi$ and $\psi$ are homomorphisms of differential graded algebras.
}%
\medskip

In particular, $\Omega CX$ is isomorphic to a sub differential graded algebra of $CGX$, and both $\phi$ and $\psi$ induce isomorphisms of algebras in homology. 

\begin{remark} Let $X$ and $Y$ be $1$-reduced simplicial sets, and $f: GX \to GY$ be a simplicial map (not necessarily a homomorphism).  The theorem above gives us a natural way to construct a chain-level model of $f$.  Indeed, if we set 
$$\xi=\psi\circ Cf\circ \phi: \Omega CX \to \Om CY,$$
then
$$\xymatrix{ \Omega CX \ar [d] _{\phi}\ar [r]^\xi& \Om CX\ar [d]^\phi\\ CGX \ar [r]^{Cf}& GCX}$$
commutes up to natural chain homotopy.
\end{remark}

\begin{remark}  It was proved in \cite {hps} that if $X$ is a
  simplicial suspension, then $\Om CX$ is naturally a primitively
  generated Hopf algebra, and the chain algebra map $\phi: \Om CX\to
  CGX$ also respects comultiplicative structure.  
\adraft{In this case}
the strong deformation retract of the theorem above is actually \emph
  {Eilenberg-Zilber data} \cite {gugenheim-munkholm}, which implies
  that the chain algebra map $\psi:CGX\to \Om CX$ is also
  \emph{strongly homotopy comultiplicative} [ibid.].  
\end{remark}

In order to prove the theorem above,  we extend the definition of the cobar construction
and actually obtain the existence of such a strong deformation retract
for all $0$-reduced simplicial sets.

The homomorphism $\phi$, which we recall in the first section of this
article, was first described by 
{Szczarba}~\cite{szczarba} in
the language of twisting cochains. Given a 
simplicial set $X$
\adraft{that is $0$-reduced but not}
necessarily 1-reduced, he gives an explicit,
though somewhat complicated, formula for a twisting cochain, 
$$
\alpha_\phi:CX\to CGX,
$$
that is based on the universal twisting function
$\tau:X\to GX$ and that gives rise in usual way to an algebra homomorphism
$$
\phi:\Omega CX\to CGX.
$$

In degree zero the cobar construction is a free associative algebra on
symbols given
by the non-degenerate 1-simplices of $X$, while the
{right-hand} side is the group ring on the free group on the same
symbols. In the first section of the paper we
observe that if $X$ is not $1$-reduced, then
we may perform a change of 
rings along $\phi_0$,
obtaining an extended cobar construction $\ccobar X$, together with an algebra homomorphism
$$
\phi:\ccobar X\to CGX
$$ 
that is an isomorphism in degree zero.

In the second section we introduce the retraction map
 $\psi$ from the chains on the loop group to the extended cobar
 construction, for which we provide an explicit recursive formula. 
We prove in {fact} that $\psi$ is a {natural} homomorphism 
of chain algebras and a one-sided inverse of the 
Szczarba map $\phi$. {It is surprising that such a map has not
  been previously observed in the literature.}

In the third section we 
{complete the strong deformation retraction (\ref{eqn:sdr}) by defining} 
the natural homotopy  $\Phi$.
{For this, we use} 
the acyclic models for the loop group on 0-reduced simplices
studied by Morace and Prout\'e~\cite{mp}.

\section{Preliminaries}

\subsection{Simplicial notions and notation}
A simplicial set $X$ is a contravariant functor from the category of finite non-empty
ordinals $\Delta$ to the category of sets; more prosaically it is a sequence of
sets $X_n$, $n\geq0$, and specified face and degeneracy operators $$
d_i:X_n\to X_{n-1},\qquad s_i:X_n\to X_{n+1},\qquad (0\le i\le n)
$$
satisfying the simplicial identities, see for example \cite{may}.
A simplicial set is $n$-reduced if $X_k\cong\{*\}$ for $k\leq n$. The
notions of simplicial group and 
\adraft{simplicial objects in other categories}
are analogous.

Given an element $x\in X_n$ and
any composite $\theta$ of simplicial face and degeneracy
 operators, represented by a monotonic function
$f:\{0,\dots,m\}\to \{0,\dots,n\}$ in $\Delta$, we also write
$$
\theta(x)\;\;=\;\;x_{(f(0),\dots,f(m))}
$$
for the corresponding element of $X_m$. We may write the face and
degeneracy maps \adraft{themselves}, for example, as
\begin{align*}
d_i(x)\;\;&=\;\;x_{(0,\dots,\widehat i,\dots,n)},
\\
s_i(x)\;\;&=\;\;x_{(0,\dots,i,i,\dots,n)}.
\end{align*}

The simplicial relations imply that any simplicial 
operator 
\adraft{$X_n\to X_m$}
has normal form 
\[\theta\;\;=\;\;s_{i_1}\ldots s_{i_q}d_{j_1}\ldots d_{j_r}\]
with $i_k>i_{k+1}$ and $j_k<j_{k+1}$ for all $k$.
In this form, the corresponding {\em derived operator} 
\adraft{$X_{n+1}\to X_{m+1}$}
is
\[\theta'\;\;=\;\;s_{i_1+1}\ldots s_{i_q+1}d_{j_1+1}\ldots
d_{j_r+1}\,.\]
An operator $\theta$ is {\em frontal\/} if it contains no
$d_0$; such operators satisfy $\theta's_0=s_0\theta$.

\subsection{The cobar construction}

We introduce a slightly extended definition of the cobar construction, 
\adraft{which will be better suited for applying}
to the normalised chain complex on a $0$-reduced simplicial set. 
Our definition generalises the classical construction of Adams, with which it agrees for simply connected coalgebras.

Let ${(C,\dd,\Delta)}$ be a connected differential graded
coalgebra over a commutative ring $R$, {so that} $C_{0}=R$.  We suppose furthermore that $C$ is $R$-free in each degree. Consider the
ring $B$ given by the free associative $R$-algebra on the desuspension of $C_1$,
\begin{align*}
B&\;\;=\;\;T(s^{-1}C_1)\;\;=\;\;\bigoplus_{k\geq0}(s^{-1}C_1)^{\otimes k},
\end{align*}
Fix an $R$-basis $\{x_{j}\mid j\in J\}$ of $C_{1}$, and let $A$ be the
ring {obtained from $B$} by freely adjoining inverses
$\invert {{j}}$
of all elements of the form $1+s^{-1}x_{j}$, for all $j\in J$. Explicitly,
\begin{align*}
A
&\;\;=\;\;T_B\bigl(\invert {{j}}\,|\,j\in J\bigr)
\;\;/\;\;\bigl(\invert {{j}}\otimes(1+s^{-1}x_{j})=1=(1+s^{-1}x_{j})\otimes\invert {{j}}\bigr).
\end{align*}
Observe that the relations may also be expressed in the form
$$
\invert j\otimes s^{-1}x_{j}\;\;=\;\;1-\invert j\;\;=\;\;s^{-1}x_{j}\otimes\invert j.
$$

The ring $A$ may be regarded as a differential graded algebra concentrated
in degree zero. The graded algebra underlying the extended cobar construction $(\ccobar ,\partial^\Omega)$ is then
\begin{align*}
\ccobar \;\;
&=\;\;T_A(s^{-1}C_{\geq2})\\
&=\;\;\bigoplus_{k\geq 0} A\otimes (s^{-1}C_{\geq2}\otimes A)^{\otimes k}.
\end{align*}
Each $R$-module $(\ccobar )_n$ is generated by
words $$a=a_1\otimes\dots\otimes a_r,\qquad |a_i|=n_i,\quad|a|=n=\sum n_i,
$$ 
where
either $a_i=s^{-1}c$ for some $R$-basis element 
$c\in \adraft{  C_{ n_{i} +1 } }$ 
or $n_i=0$ and $a_i=\invert{j}$ for some $j\in J$. 
This $R$-module is free  on those words in which 
$\invert j$ does not appear adjacent to $s^{-1}x_{j}$.
The algebra multiplication is induced by concatenation of words,
extended bilinearly to $\ccobar $, modulo the relation that
$\invert j$ is inverse to $1+s^{-1}x_{j}$.

The differential $\partial^\Omega$ on $\ccobar$ is the derivation of $A$-algebras
that is specified by
\begin{eqnarray}
\nonumber
\partial_n^\Omega s^{-1}c&=&
- s^{-1}dc
\;\;+\;\;
{(-1)^{|c_{i}|}} s^{-1}c_{i}\otimes s^{-1} c^{i}
\end{eqnarray}
for all basis elements {$c\in C_{n+1}$ and all $n\geq 1$}, where $\Delta(c)= 1\otimes c+ c\otimes 1 + c_{i}\otimes c^{i}$ (using the Einstein summation convention). Note that $\partial^\Omega$ is necessarily zero on elements of $A$.

The unit $1\in(\ccobar )_0$ is identified with the empty word, via
the
isomorphism $R\cong ({s^{-1}C_1})^{\otimes 0}$.

\begin{remark} If $C$ is simply connected, so that $C_1=\{0\}$ and therefore  $A=B=R$,
then $\ccobar $ coincides with the usual cobar contruction $\Omega C$
defined by Adams.
\end{remark}

\subsection{The Kan loop group}\label{szsection}

Let $X$ be a $0$-reduced simplicial set and $G$ a simplicial group.
A {\em twisting function} $\tau:X\to G$ is a collection of 
functions  of  degree $-1$
$$\{\tau:X_{n+1}\to G_n\mid n\ge0\}$$
satisfying
\begin{align}
\tau s_0x&=1\,, \nonumber
\\
s_i\tau x&=\tau s_{i+1}x \label{twisting2}\,,\\
d_0\tau x&={\tau d_0x}^{-1}\cdot \tau d_1x\,,\\
d_i\tau x&=\tau d_{i+1}x \quad\textrm{ if }i\ge1
\,.\label{twistinglast}
\end{align}

Let $GX$ denote  the Kan loop group on $X$, which is a simplicial
group that models the space of based loops on the geometric
realization of $X$ {(see \cite{kan,may})}.  There is a {\em universal} twisting function
\begin{align*}
\tau:X_{n+1}&\twoheadrightarrow(GX)_n=
{\mathbf F}(X_{n+1})/{\mathbf F}(s_0X_n)
\\
{\tau x}&{\;\,=\;\,} [x]
\end{align*}
sending $x\in X_{n+1}$ to the class of the corresponding generator in
quotient of free groups,
{and the  simplicial structure on $GX$ is defined by
\eqref{twisting2}--\eqref{twistinglast}}.

Recall that the normalised chain complex $CG$ on a simplicial group
$G$ also has a
differential graded algebra structure, with multiplication given by
the shuffle map and the multiplication in $G$,
$$
m:CG\otimes CG\longrightarrow C(G\times G)\longrightarrow CG,
$$
that is,
\begin{equation}\label{m-shuffle}
m(g\otimes h)=
\sum
(-1)^{\sgn{(\mu,\nu)}}
s_{\mu(q)}\dots s_{\mu(1)} g \cdot s_{\nu(p)}\dots s_{\nu(1)} h,
\quad g\in G_p,h\in G_q,
\end{equation}
where the summation is over all $(p,q)$-shuffles {$(\mu,\nu)\in\Shuff(p,q)$}.

The following proposition is the motivation for our extension of the cobar
construction. Recall that for any simplicial set $X$, the degree $n$
part of its normalised chain complex, $C_{n}X$, is the free abelian
group on the set of all nondegenerate $n$-simplices of $X$
\adraft{and that $CX$ has a comultiplication $\Delta:CX\to CX\otimes
  CX$ given by the Alexander-Whitney diagonal approximation 
\begin{align*}
\Delta(x)&\;=\;x_{(0)}\otimes x\;+\;
\sum_{i=1}^{n-1}x_{(0,\dots,i)}\otimes x_{(i,\dots,n)} 
\;+\;x\otimes x_{(n)}
\end{align*}  
for $x\in X_n$, $n\geq1$, and with $\Delta(x)=x\otimes x$ for $x\in
X_0$.} 
In particular, if $X$ is $0$-reduced, then $CX$ is a connected,
differential graded coalgebra over $\mathbb Z$. 
 
\begin{proposition}\label{degreezero}
Let $X$ be a $0$-reduced simplicial set and $GX$ its Kan loop
group. Then there is an isomorphism of rings
\[
\xymatrix{(\ccobar X)_0\rto<1.1ex>_-\cong^-{\phi_{0}}&(CGX)_0\lto<1.1ex>^-{\psi_{0}} 
}
\]
determined by
\begin{align*}
\psi_0(\tau x)&=\invert x&
\psi_0(\tau x^{-1})&=1+s^{-1}x\\
\phi_0(s^{-1}x)&=\tau x^{-1} - 1&
\phi_0(\invert x)&=\tau x.
\end{align*}
\end{proposition}
\begin{proof}
The proof is straightforward. Note that if $x$ is the degenerate element $s_0(*)$
then the four equations say $\psi_0(1)=1$, $\phi_0(0)=0$,
$\phi_0(1)=1$. In degree 0 the multiplication \eqref{m-shuffle} is
just $m(g\otimes h)=g\cdot h$, and we have 
\begin{align*}
\psi_0(\tau x_1^{\alpha_1}\dots\tau x_r^{\alpha_r})&=
\psi_0(\tau x_1^{\alpha_1})\otimes\dots\otimes \psi_0(\tau x_r^{\alpha_r})\\
\phi_0(a_1\otimes \dots\otimes a_r)&=
\phi_0(a_1)\cdot\dots\cdot\phi_0(a_r)
\end{align*}
where $x_i\in X_1$, $\alpha_i=\pm1$ and $a_i=s^{-1}x_i$ or
$\invert{x_i}$. This is well defined: $\psi_{0}(g)$ is inverse to $\psi_{0}(g^{-1})$ for all $g\in (GX)_{0}$
and $\phi_0(\invert x)$ is inverse to $\phi_0(1+s^{-1}x)$ for all $x\in X_{1}\smallsetminus \{s_{0}(*)\}$. It is also
clear that the composites $\phi_0\psi_0$ and $\psi_0\phi_0$ are the
respective identity maps:
\begin{align*}
\phi_0\psi_0(\tau x)&=\phi_{0}(\invert x)=\tau x&
\phi_0\psi_0(\tau x^{-1})&=\phi_0(1+s^{-1}x)=\tau x^{-1}\\
\psi_0\phi_0(s^{-1}x)&=\psi_0(\tau x^{-1} - 1)=s^{-1}x&
\psi_0\phi_0(\invert x)&=\psi_0(\tau x)=\invert x.
\end{align*}
\end{proof}

\subsection{The Szczarba map}

A map of differential graded algebras
$$
\phi: \adraft{ \Omega C } X \to CGX
$$
was given explicitly by Szczarba in the language of twisting
cochains.
The following Definition, Lemma and Theorem are from sections 2 and 3 of 
Szczarba's paper~\cite{szczarba}, adapted slightly to 
\adraft{define a map on the extended cobar construction
$$
\phi:\ccobar X\to CGX
$$
that extends the isomorphism}
$\phi_0$ of Proposition \ref{degreezero}.

Let $S_n$ be the set of $n!$ sequences of integers $$i=(i_1,\ldots,i_n)
\quad\text{ such that }\; 0\leq i_k\leq n-k\text{ for each }{k}.$$ 
In particular, $i_n=0$. 
The {\em sign} of such a sequence $i\in S_n$ is 
$$
(-1)^{\sum i},\quad\text{ where }\quad{\textstyle\sum i}\;=\;i_1+i_2+\cdots+i_n\,.
$$

\begin{definition}\label{szdef}
Given a twisting function $\tau:X\to G$, the \,{\em Szczarba operators}
are the functions
\[\Sz_i:X_{n+1}\longrightarrow G_n\,,\qquad 
\adraft{i=(i_1,\ldots,i_n)}
\in S_n\,,\]
given by the following product in $G_n$,
\begin{eqnarray*}
\Sz_i\,x&=&
D_{0;\,i}^{n+1}
\;{\tau x}^{-1}\cdot
D_{1;\,i}^{n+1}\;{\tau d_0x}^{-1}\cdot\,\cdots\,\cdot
D_{n;\,i}^{n+1}\;{\tau d_0^nx}^{-1}\,.
\end{eqnarray*}
Here the operators $D_{j;\,i}^{n+1}:G_{n-j}\to G_n$ for $i\in S_n$,
$j=0,\ldots,n$, are defined as
\begin{eqnarray}
D_{0;\,()}^1&=&
\adraft{Id_{G_0},}
\nonumber\\ \label{dji}
D_{j;\,i_1,\ldots,i_n}^{n+1}&=&
\left\{\displaystyle\begin{array}{ll}
D^{n\,\,\prime}_{j;\,i_2,\ldots,i_n}
\,s_0\,d_{i_1-j}&\textrm{if }j<i_1\,,
\\[1.3ex]
D^{n\,\,\prime}_{j;\,i_2,\ldots,i_n}&\textrm{if }j=i_1\,,
\\[1.3ex]
D^{n\,\,\prime}_{j-1;\,i_2,\ldots,i_n}\,s_0&\textrm{if }j>i_1\,.
\end{array}\right.
\end{eqnarray}
\end{definition}
As simplicial operators these are all frontal:
\adraft{defining 
$D_{j;\,i}^{n+1}\colon X_{n-j}\to X_n$ in the same way,
one has
$D_{j;\,i}^{n+1}\tau=\tau\,{D_{j;\,i}^{n+1}}'\colon X_{n-j+1}\to G_n$.}

\begin{lemma}\label{szrels}
The Szczarba operators satisfy
\par\smallskip\noindent
\[\begin{array}{lcl}
d_0\,\Sz_{i_1,\ldots,i_n}
&=&\qquad\quad
\Sz_{i_2,\ldots,i_n}\,d_{i_1+1}\,,
\\[1.3ex]
d_k\,\Sz_{i_1,\ldots,i_k,i_{k+1},\ldots,i_n}
&=&\quad
d_k\,\Sz_{i_1,\ldots,i_{k+1},i_k-1,\ldots,i_n}\;\;
\textrm{ if }\;i_k>i_{k+1}\,,
\\[1.3ex]
d_n\,\Sz_{i_1,\ldots,i_n}\,x&=&
s_{\mu}\,\Sz_{i'
}\,x_{(0,\ldots,r)}
\,\cdot\,
s_{\nu}\,\Sz_{i''
}\,x_{(r,\ldots,n+1)}\,.
\end{array}\]
{In the last equation the sequences $i'$, $i''$, the integer $r$} and the $({r-1,n-r})$-shuffle
$(\mu,\nu)$ are defined by a certain bijection
\begin{eqnarray*}
S_n&\cong&\bigcup_{r=1}^n \;\Shuff({r-1,n-r})
\times S_{r-1}\times S_{n-r}
\\
i\;&\mapsto&\qquad\qquad\;(\quad(\mu,\,\nu)\;\;,\qquad i'\,,\qquad i''\quad)\,,
\end{eqnarray*}
see~\cite[{\sc  Lemma 3.3}]{szczarba}, which respects
parity as follows:
\begin{eqnarray*}
{\textstyle n+\sum i}&=&{\textstyle r+\sgn(\mu,\nu)+\sum i'+\sum i''}\pmod2\,.
\end{eqnarray*}
\end{lemma}
Note that Szczarba's sign conventions differ slightly from ours, and that his
inductively-defined parity $\epsilon(i,n+1)$ is in fact just 
$n+\sum i$.
\medskip
\begin{theorem}\label{thm:defn-phi}
For any twisting function $\tau:X\to G$ on a 0-reduced simplicial set $X$
there is a canonical homomorphism of differential graded algebras 
\begin{eqnarray*}
\phi\; :\;\ccobar X&\longrightarrow&C G\\
\textrm{defined by}\qquad\qquad\qquad
\phi_0 (\invert {x_1})&=& \tau x_1
\\
\phi_0 (s^{-1}x_1)&=& \tauinv{x_1}\;-\;1
\\
\phi_n (s^{-1}x_{n+1})&=&\sum_{i\in S_n}(-1)^{\sum i}\;\;\Sz_i\,x\,,\quad(n\ge1)\,,
\end{eqnarray*}
for $x_{n+1}\in X_{n+1}$.
\end{theorem}
\begin{proof}
The map $\phi$ extends linearly, and multiplicatively via 
$$\phi_{p+q}(a\otimes b)=m(\phi_p(a)\otimes\phi_q(b)),\qquad |a|=p,|b|=q,$$ 
where $m$ is the multiplication \eqref{m-shuffle}, to all of $\ccobar X$.
We show $\phi$ is a chain map, i.e., that 
$\partial_n\phi_n=\phi_{n-1}\partial^\Omega _n$. 
For $x\in X_2$ we can write
\begin{align*}
\partial_1^\Omega s^{-1}x
\;&=\;-s^{-1}d_0x+s^{-1}d_1x-s^{-1}d_2x-s^{-1}x_{(0,1)}\otimes
s^{-1}x_{(1,2)}
\\ &=\;
(1+s^{-1}d_1x)-(1+s^{-1}x_{(0,1)})\otimes(1+s^{-1}x_{(1,2)})
\end{align*}
and so we have, by Lemma \ref{szrels},
\begin{align*}
\partial_1\phi_1s^{-1}x&\;=\;\partial_1\Sz_{0}x
\;=\;d_0\Sz_{0}x-d_1\Sz_{0}x\;=\;\Sz_{()}d_1x-\Sz_{()}x_{(0,1)}\cdot\Sz_{()}x_{(1,2)}\\
&\;=\;\tauinv{d_1x}-\tauinv{x_{(0,1)}}\cdot\tauinv{x_{(1,2)}}
\;=\;\phi_0\partial_1^\Omega s^{-1}x.
\end{align*}
For $x\in X_{n+1}$ the argument is essentially the same. We have
\begin{eqnarray*}
\partial_n\,\phi_n\,s^{-1}x&=&
\sum_{i\in S_n}(-1)^{\sum i}\;\partial_n\,\Sz_i\,x\\
&=&\sum_{i\in S_n}\;(-1)^{\sum i}\;
\biggl(\,\sum_{k=0}^{n}(-1)^{k}\;d_k\,\Sz_i\,x
\biggr)
\end{eqnarray*}
where, by Lemma~\ref{szrels}, all the terms for $0<k<n$ cancel, and
the terms for $k=0,n$ may be rewritten as
\begin{eqnarray*}
&&
\!\!\!\!\!\!\!\!\!
\sum_{
\begin{subarray}{c}
\!\!\!\!
0\le i_1\le n-1
\!\!\!
\\[0.6ex]
\!\!\!
i\in S_{n-1}
\end{subarray}
}\!\!(-1)^{i_1+\sum i}\,\Sz_i\,d_{i_1+1}\,x
+\!\!
\sum_{
\begin{subarray}{c}
\!\!\!\!
1\le r\le n
\!\!\!
\\[0.6ex]
\!\!\!\!
(\mu,\nu),\, i',i''
\!\!\!\!
\end{subarray}
}
(-1)^{n+\sum i}
s_{\mu}\,\Sz_{i'}\,x_{(0,\ldots,r)}
\cdot
s_{\nu}\,\Sz_{i''}\,x_{(r,\ldots,n+1)}
\\&=&
\!\!\!
\sum_{\!\!\!\!1\le r\le n\!\!\!}
(-1)^{r-1}\; \phi_{n-1}\,s^{-1}d_rx
\,+\,
\sum_{\!\!\!\!1\le r\le n\!\!\!}
(-1)^{r}\;m(
\phi_{r-1}\,s^{-1}x_{(0,\ldots,r)}\otimes 
\phi_{n-r}\,s^{-1} x_{(r,\ldots,n+1)})
\\&&\qquad\qquad\qquad\qquad\qquad
{}-1\otimes \phi_{n-1}s^{-1}x_{(1,\dots,n+1)}
+(-1)^n\phi_{n-1}s^{-1}x_{(0,\dots,n)}\otimes 1
\\&=&
\phi_{n-1}\biggl(
\;{ 
\sum_{r=0}^{n+1}(-1)^{r-1} s^{-1}{d_rx}
\;\;+\;\;
\sum_{r=1}^n(-1)^r s^{-1}x_{(0,\ldots,r)}\otimes s^{-1}x_{(r,\ldots,n+1)}
}\biggr)
\\&=&
\phi_{n-1}\partial^\Omega _n\,x\,.
\end{eqnarray*}
\end{proof}

We will need one further property of the Szczarba operators.
\begin{lemma}\label{commondeg}
For all $x\in X_{n+1}$ and $i\in S_n$, the following product in $G_n$
is degenerate:
\begin{eqnarray*}
D_{0;\,i}^{n+1}
\;{\tau x}\cdot
\Sz_i\,x&=&
D_{1;\,i}^{n+1}\;{\tau d_0x}^{-1}\cdot\,\cdots\,\cdot
D_{n;\,i}^{n+1}\;{\tau d_0^nx}^{-1}\,.
\end{eqnarray*}
\end{lemma}
\begin{proof}
For any sequence $i\in S_n$ we will show by induction 
that $D^{n+1}_{j;\,i}$ 
is $s_{\kappa(i)-1}$-degenerate for all
$j>0$, where $\kappa(i)$ is the least integer such that $i_{\kappa(i)}=0$. If  $i_1=0$,
so that $\kappa(i)=1$, then by~(\ref{dji})
$$D^{n+1}_{j;\,i}\;\;=\;\;
D^{n\,\,\prime}_{j-1;\,i_2,\ldots,i_n}\,s_0\;\;=\;\;
s_0D^{n}_{j-1;\,i_2,\ldots,i_n}$$
for all $j>i_1=0$.
If  $i_1>0$, so that $\kappa(i)>1$, then
$\kappa(i_2,\ldots,i_n)=\kappa(i)-1$ and
we know that $D^n_{j;\,i_2,\ldots,i_n}$ (if $j>0$) and
$D^n_{j-1;\,i_2,\ldots,i_n}$ (if $j>1$)
are $s_{\kappa(i)-2}$-degenerate 
by the inductive hypothesis. The
corresponding derived operators are therefore $s_{\kappa(i)-1}$-degenerate
and, by~(\ref{dji}), 
so is $D^{n+1}_{j;\,i}$ for all $j>0$.  
\end{proof}

\section{The retraction map}

\subsection{Definition of the map}

Let $X$ be a $0$-reduced simplicial set. We introduce in this section a
map of differential graded algebras
$$
\psi:CGX\longrightarrow \ccobar X
$$
between the chains on the loop group and the extended cobar
construction, which is  a retraction of the Szczarba map $\phi$.  The map $\psi$ is uniquely determined by the relation
\begin{eqnarray}\label{psixg}
\psi_n(\tau x\cdot g)
&=&
\psi_n(g) \;- \sum_{i=0}^{n} s^{-1}x_{(0,\dots,i+1)}\otimes \psi_{n-i}(\tau d_1^ix\cdot d_0^ig)
\end{eqnarray}
for $x\in X_{n+1}$ and $g\in (GX)_n$.
Note that the $i=0$ term on the {right-hand} side is
$s^{-1}x_{(0,1)}\otimes \psi_n(\tau x\cdot g)$. 
In fact $\psi$ may be expressed inductively, on the degree $n$ and the
word length in $(GX)_n$.

\begin{lemma}\label{inducform}
The definition of $\psi$ in \eqref{psixg} may be rewritten as
\begin{align*}
&\psi_n(\tau x\cdot g)=
\invert{x_{(0,1)}}
\otimes\biggl(\psi_n(g) -
\sum_{i=1}^{n} s^{-1}x_{(0,\dots,i+1)}\otimes \psi_{n-i}\finalom ixg\biggr),
\\
&\psi_n(\tau x^{-1}\cdot h)=
(1+s^{-1}x_{(0,1)})\otimes\psi_n(h) + 
\sum_{i=1}^{n} s^{-1}x_{(0,\dots,i+1)}\otimes \psi_{n-i}\finalombar ixh .
\end{align*}
where
\begin{align*}
\finalom ixg&=\tau d_1^ix\cdot d_0^ig\;\in\;(GX)_{n-i},\\
\finalombar ixh&=\finalom ix{\tau x^{-1}\cdot h}=\tau d_0d_2^{i-1}x\cdot d_0^ih.
\end{align*}
\end{lemma}
\begin{proof}
Collecting the terms in \eqref{psixg} involving $\psi_n(\tau x\cdot g)$
and dividing by $1+s^{-1}x_{(0,1)}$ gives the first equation.
The second is obtained by taking $g=\tau x^{-1}\cdot h$ in the first.
\end{proof}
From these formulae it is straightforward to give the map $\psi$
explicitly in low degrees.
\begin{lemma}
The map
$\psi_0:(GCX)_0\to(\ccobar X)_0$ agrees with that defined in
Proposition \ref{degreezero}, and the map
$\psi_1:(GCX)_1\to(\ccobar X)_1$ is given
for $x,x_i\in X_2$ and $\alpha,\alpha_i=\pm1$ 
by
\begin{eqnarray*}
\psi_1(\tau x_1^{\alpha_1}\dots\tau x_r^{\alpha_r})&=&
\sum_{i=1}^r
\psi_0d_1(\tau x_1^{\alpha_1}\dots\tau x_{i-1}^{\alpha_{i-1}})
\otimes \psi_1(\tau x_i^{\alpha_i})         \otimes
\psi_0d_0(\tau x_{i+1}^{\alpha_{i+1}}\dots\tau x_{r}^{\alpha_{r}})
\\\textrm{with }\psi_1(\tau x^{\alpha})
&=&\left\{\begin{array}{rl}
-\psi_0(\tau x_{(0,1)})\otimes s^{-1}x\otimes\psi_0(\tau x_{(0,2)})
&(\alpha=+1)\\
s^{-1}x\otimes\psi_0(\tau x_{(1,2)})
&(\alpha=-1).
\end{array}\right.
\end{eqnarray*}
\end{lemma}

\begin{lemma}
$\psi$  is well-defined.
\end{lemma}
\begin{proof}
We show that $\psi_n(w)=0$ if $w$ is degenerate, by induction on $n$ and on word length in $GX$.
Suppose  $0\le j\le n-1$ and
$$w=s_j(\tau x^\alpha\cdot g),$$ 
where $\alpha=\pm1$ and $\tau x^\alpha\cdot g$ is a reduced word in $(GX)_{n-1}$.
For $\alpha=+1$ we have
\begin{align*}
(1+s^{-1}x_{(0,1)}&)\otimes\psi_n(w)=(1+s^{-1}x_{(0,1)})\otimes\psi_n(\tau s_{j+1}x\cdot s_jg)
\\
&=\psi_ns_jg-\sum_{i=1}^{n}(s_{j+1}x)_{(0,\dots,i+1)}
\otimes\psi_{n-i}\finalom i{s_{j+1}x}{s_jg}
\end{align*}
in which the first term is zero inductively. 
Each term in the summation is also zero since 
$(s_{j+1}x)_{(0,\dots,i+1)}$ is degenerate for $j<i$ and
$\finalom i{s_{j+1}x}{s_jg}=s_{j-i}\finalom ixg$ for $j\geq i$.
Since $1+s^{-1}x_{(0,1)}$ is invertible, we have $\psi_n(w)=0$. 

The argument for $\alpha=-1$ is similar.
\end{proof}

\subsection{Properties of the retraction map}

We now prove that $\psi$ is a morphism of differential graded algebras and a retraction of the Szczarba map $\phi$.

\begin{proposition}\label{chnmap} 
$\psi$
is  a chain map.
\end{proposition}

\begin{proof}
We will show that for all $x\in X_{n+1}$ and $g\in (GX)_n$,
\[
\partial_n^\Omega\,\psi_n
(\tau x\cdot g)
\;\;=\;\;\psi_{n-1}\,\partial_n(\tau x\cdot g),
\]
by induction on $n$ and on word length in $GX$. We first observe that
\begin{align*}
&\psi_{n-1}(d_0(\tau x\cdot g)-\tau d_1x\cdot d_0g)\;=\;
\psi_{n-1}(\tau d_0x^{-1}\cdot\tau d_1x\cdot d_0g)-\psi_{n-1}(\tau d_1x\cdot d_0g)
\\
&=s^{-1}x_{(1,2)} \otimes \psi_{n-1}(\tau d_1x\cdot d_0g)
+ \sum_{i=1}^{n-1} s^{-1}x_{(1,\dots,i+2)} \otimes \psi_{n-1-i}\finalombar i{d_0x}{\tau d_1x\cdot d_0g},
\end{align*}
using the second formula in Lemma~\ref{inducform}. 
Now since 
$\tau d_1x\cdot d_0g=\finalom1xg$, and also
$\finalombar i{d_0x}{\tau d_1x\cdot d_0g} =
\finalombar i{d_1x}{\tau d_1x\cdot d_0g} =\finalom i{d_1x}{d_0g}=\finalom {i+1}xg$, 
we get
\begin{align}
\psi_{n-1}(d_0(\tau x\cdot g)-\tau d_1x\cdot d_0g)&\;=\;
\sum_{k=1}^{n} s^{-1}x_{(1,\dots,k+1)}\otimes\psi_{n-k}\finalom kxg
\label{rel1}
\end{align}
and, substituting $d_1^{r-1}x$ and $d_0^{r-1}g$ for $x$ and $g$
respectively, we obtain
\begin{align}
\psi_{n-r}(d_0\finalom {r-1}xg
-\finalom rxg)&\;=\;\label{rel2}
\sum_{k=r}^{n} s^{-1}x_{(r,\dots,k+1)}\otimes\psi_{n-k}\finalom kxg.
\end{align}
Now, using Lemma~\ref{inducform}, we know that
\begin{align*}
(1+s&^{-1}x_{(0,1)})\otimes\partial^\Omega_n\psi_n(\tau x\cdot g)
\;\;=\;\;\partial^\Omega_n\bigl((1+s^{-1}x_{(0,1)})\otimes \psi_n(\tau x\cdot g)\bigr)
\\
=&\;\;
\partial^\Omega_n\biggl(
-
\sum_{i=1}^ns^{-1}x_{(0,\dots,i+1)}
\otimes \psi_{n-i}\finalom ixg
\;\;+\psi_ng\biggr)
\\
=&\;-\;\sum_{i=1}^n\partial^\Omega_is^{-1}x_{(0,\dots,i+1)}\otimes\psi_{n-i}\finalom ixg\\
\;-\;&\sum_{i=1}^{n-1}(-1)^is^{-1}x_{(0,\dots,i+1)}\otimes\psi_{n-1-i}\partial_{n-i}
\finalom ixg\;\;+\;\;\psi_{n-1}\partial_ng,
\end{align*}
since inductively 
$\partial^\Omega_{n-i}\psi_{n-i}=\psi_{n-i-1}\partial_{n-i}$
and
$\partial_n^\Omega\psi_ng=\psi_{n-1}\partial_ng$. Now expanding the
operators $\partial^\Omega$ and $\partial$ we get
\begin{align}\nonumber
(&1+s^{-1}x_{(0,1)})\otimes\partial^\Omega_n\psi_n(\tau x\cdot g)
\;\;=\\\nonumber 
&
\;-\;\sum_{i=1}^n
\sum_{r=0}^{i+1}
(-1)^{r+1}s^{-1}x_{(0,\dots,\hat r,\dots i+1)}\otimes\psi_{n-i}\finalom ixg
\\\label{big}
&\;-\;\sum_{i=1}^n
\sum_{q=1}^{i}(-1)^{q}s^{-1}x_{(0,\dots,q)}\otimes s^{-1}x_{(q,\dots,i+1)}\otimes\psi_{n-i}\finalom ixg
\\\nonumber
&\;-\;
\sum_{j=1}^{n-1}
\sum_{t=0}^{n-j}
(-1)^{j+t} s^{-1}x_{(0,\dots,j+1)}\otimes\psi_{n-1-j}d_t
\finalom jxg\;+\;\sum_{k=0}^n(-1)^k\psi_{n-1}d_kg.
\end{align}
Collecting together the terms for which either $i=1$, $k=0$, $r=0$, or $q=1$ gives
\begin{align*}
\biggl((&1+s^{-1}x_{(0,1)})-(1+s^{-1}x_{(0,2)})\biggr)\otimes \psi_{n-1}\finalom1xg\;\;\;+\;\;\psi_{n-1}d_0g
\\+&
\sum_{i=1}^n (1+s^{-1}x_{(0,1)})\otimes s^{-1}x_{(1,\dots i+1)}\otimes\psi_{n-i}\finalom ixg
\\
=\;&(1+s^{-1}x_{(0,1)})\otimes\psi_{n-1}d_0(\tau x\cdot g)
-(1+s^{-1}x_{(0,2)})\otimes \psi_{n-1}\finalom1xg
+\psi_{n-1}d_0g
\end{align*}
by \eqref{rel1}, and by Lemma \ref{inducform} the last two terms here
cancel exactly with the terms for $r=1$ and $i>1$ in \eqref{big}.

Now collecting the terms for $r=i+1$ and $i>1$ in \eqref{big},
together with all the $(i,q)$-indexed terms not already considered, gives
\begin{align*}
&-
\sum_{i=2}^n
\biggl(
(-1)^is^{-1}x_{(0,\dots,i)}
+\sum_{q=2}^i(-1)^{q}s^{-1}x_{(0,\dots,q)}\otimes s^{-1}x_{(q,\dots,i+1)}
\biggr)\otimes\psi_{n-i}\finalom ixg
\\
&=-
\sum_{r=2}^{n}(-1)^{r}
s^{-1}x_{(0,\dots,r)}
\otimes 
\biggl(\psi_{n-r}\finalom{r}xg
+\sum_{k=r}^n
s^{-1}x_{(r,\dots,i+1)}
\otimes\psi_{n-i}\finalom ixg
\biggr)
\\
&=-
\sum_{r=2}^{n}(-1)^{r}
s^{-1}x_{(0,\dots,r)}
\otimes \psi_{n-r}d_0\finalom {r-1}xg
\end{align*}
by \eqref{rel2}, and this cancels exactly  with the terms for $t=0$ in
\eqref{big}.

Thus expression \eqref{big} is equal to
\begin{align}\nonumber
&(1+s^{-1}x_{(0,1)})\otimes\psi_{n-1}d_0(\tau x\cdot g)
\;+\;\sum_{k=1}^n(-1)^k\psi_{n-1}d_kg\\\label{big2}
&
\;-\;\sum_{2\le r\le i\le n}
(-1)^{r+1}s^{-1}x_{(0,\dots,\hat r,\dots i+1)}\otimes\psi_{n-i}\finalom ixg
\\\nonumber
&\;-\;
\sum_{1\le j<j+t\le n}
(-1)^{j+t} s^{-1}x_{(0,\dots,j+1)}\otimes\psi_{n-1-j}d_t
\finalom jxg.
\end{align}
Now to complete the proof it remains to show that this is equal to
\begin{align*}
(1+s^{-1}x_{(0,1)})\otimes
\psi&_{n-1}\partial_n(\tau x\cdot g)
\;\;=\;\;(1+s^{-1}x_{(0,1)})\otimes\psi_{n-1}d_0(\tau x\cdot g)\\
&+\sum_{i=1}^n (-1)^i(1+s^{-1}x_{(0,1)})\otimes
\psi_{n-1}(\tau d_{i+1}x\cdot d_ig).
\end{align*}
The first term is as required, and by Lemma \ref{inducform} the
summation is
$$
\sum_{i=1}^n
(-1)^i\biggl(\psi_{n-1}d_ig-\sum_{k=1}^{n-1}s^{-1}(d_{i+1}x)_{(0,\dots,k+1)}\otimes
\psi_{n-1-k}\finalom k{d_{i+1}x}{d_ig}
\biggr).
$$
The result therefore follows from the observations that
$$
(d_{i+1}x)_{(0,\dots,k+1)}\!=\!\begin{cases}x_{(0,\dots,\widehat{i+1},\dots,k+2)}
&\\x_{(0,\dots,k+1)}&\end{cases}\;\;
\finalom k{d_{i+1}x}{d_ig}\!=\!\begin{cases}\finalom{k+1}xg&\;\;(i\le
  k),\\d_{i-k}\finalom kxg&\;\;
(i>k).\end{cases}
$$
\end{proof}

\begin{proposition}
The map $\psi$ is an algebra homomorphism.
\end{proposition}
\begin{proof}
Let $v\in (GX)_p$ and $w\in(GX)_q$ and consider $v\otimes
w\in (C GX\otimes C GX)_n$, $n=p+q$.
To show that $\psi$ is multiplicative we must prove
$$\psi_n m(v\otimes w)\;=\;
\sum_{(\mu,\nu)}(-1)^{\sgn(\mu,\nu)}
\psi_{n}(s_\mu v\cdot s_\nu w)\;=\;
\psi_p v\otimes \psi_q w$$ 
in
$\ccobar X$, by induction on $p$ and the word length of $v$.
For $v=*$, or $p$ or $q=0$,
there is nothing to prove; suppose inductively that $p,q\ge1$ and $v=\tau x\cdot g$
for $x\in X_{p+1}$ (the
argument for $v=\tau x^{-1}\cdot g$ is similar). 
Then by Lemma \ref{inducform},
\begin{align*}
(1+s&^{-1}x_{(0,1)})\otimes
\psi_{n}(s_\mu (\tau x\cdot g)\cdot s_\nu w)
\;=\;
(1+s^{-1}x_{(0,1)})\otimes
\psi_{n}(\tau s_\mu'x\cdot (s_\mu g\cdot s_\nu w))
\\
=\;&\psi_{n}(s_\mu g\cdot s_\nu w)
-\!\sum_{i=1}^{n}
d_{i+2}^{n-i}s_\mu'x
\otimes \psi_{n-i}\finalom i{s_\mu'x}
{s_\mu g\cdot s_\nu w}
\end{align*}
The term $d_{i+2}^{n-i}s_\mu'x$ will be degenerate unless $i\le p$ and
$(s_\mu,s_\nu)$ is of the form
$(s_{i+\xi_q}s_{i+\xi_{q-1}}\ldots s_{i+\xi_1},\;s_0^is_\zeta)$ for
some $(p-i,q)$-shuffle $(\xi,\zeta)$, and we have
\begin{align*}
(1&{}+s^{-1}x_{(0,1)})
\otimes\psi_{n}m(v\otimes w)
\\=&
\sum_{(\mu,\nu)}(-1)^{\sgn(\mu,\nu)} \psi_{n}(s_\mu g\cdot s_\nu w)
-\!\!\sum_{\begin{subarray}{c}1\le i\le p\\(\xi,\zeta)\end{subarray}}\!
(-1)^{\sgn(\xi,\zeta)}d_{i+2}^{p-i}x\otimes \psi_{n-i}(s_\xi\finalom ixg\cdot s_\zeta w)
\\=&\;\;
\psi_nm(g\otimes w)
\;\;-\;\;\sum_{i=1}^p
d_{i+2}^{p-i}x
\;\otimes\; \psi_{n-i}m(\finalom ixg\otimes w)
\\=&\;\;
\biggl(\psi_pg
\;-\;\sum_{i=1}^p
d_{i+2}^{p-i}x
\;\otimes\; \psi_{p-i}\finalom ixg\biggr)\;\otimes\;\psi_q w,\quad \text{ by the
  inductive hypothesis,}
\\=&\;\;\;
(1+s^{-1}x_{(0,1)})\;\otimes\;\psi_p v\;\otimes\;\psi_q w,
\end{align*}
by Lemma \ref{inducform}. The result follows.
\end{proof}

\begin{proposition}\label{psiphi}
The map $\psi$ is a retraction of $\phi$, that is, the composite $\psi\phi$ is the identity.
\end{proposition}
\begin{proof}
It is enough to prove this on algebra generators of $\ccobar X$.
For $x\in X_{n+1}$ and $i=(i_1,\dots,i_n)\in S_n$ we will show that
$$
\psi_n\Sz_ix\;\;=\;\begin{cases}x&\text{ if }i_1=\dots=i_n=0, 
\\ 0&\text{ otherwise.}\end{cases}
$$
Denote by $x_{0;\,i}$ the element of $X$ satisfying
$$D_{0;\,i}^{n+1}\tau x \;\;=\;\,\tau x_{0;\,i}.$$
Lemma~\ref{commondeg} 
tells us $\psi_n(\tau x_{0;\,i}\cdot\Sz_ix)=0$,
and so by Lemma~\ref{inducform} we have 
\begin{eqnarray*}
\psi_n\Sz_ix 
&=& \sum_{k=1}^{n}
d_{k+2}^{n-k}x_{0;\,i}
\otimes \psi_{n-k}\finalom k{x_{0;\,i}}{\Sz_ix}.
\end{eqnarray*}
From (\ref{dji}) we see that $D^{n+1}_{0,\;i}\tau x$ has an $s_{k-1}$ degeneracy
if $i_k\neq0$. Thus $d_{k+2}^{n-k}x_{0;\,i}$ is degenerate except in the case
$i_1=\ldots=i_k=0$. In this case we see from~(\ref{dji}) and Lemma~\ref{szrels} that
\begin{eqnarray*}
\finalom k{x_{0;\,i}}{\Sz_ix} &=&
\tau d_1^kx_{0;\,i}\;\cdot \;d_0^i\Sz_ix\;\;=\;\;
D_{0;\,i_{k+1},\ldots,i_n}^{n-k+1}\tau d^k_1x\cdot
\Sz_{i_{k+1},\ldots,i_n}d_1^kx
\end{eqnarray*}
which is degenerate again by Lemma~\ref{commondeg}. The only non-zero
term is therefore
\begin{eqnarray*}
\psi_n\Sz_{0,\dots,0}x
&=&
x_{0;0,\dots,0}\otimes \psi_0(*)\;\;=\;\;x,
\end{eqnarray*}
and hence $\psi\phi x=x$ as required.
\end{proof}

\section{Deformation retraction of the loop group}

Both the Kan functor $G$ and the cobar construction model
loop spaces.
In the 1-reduced case it is easy to show that 
the Szczarba map 
$\phi :{\Omega C} X\to{C} GX$ is a weak equivalence, by 
applying Zeeman's comparison theorem to the 
map of spectral sequences associated with
\begin{equation}\label{sesmap}
\begin{array}[c]{c}
\xymatrix{
\Omega CX\dto_\phi\rto&\Omega CX\otimes_{\alpha_\phi}CX
\rto\dto&CX\dto^=\\
CGX\rto&C(GX\times_\tau X)\rto&CX
}
\end{array}\end{equation}
in which the total spaces are acyclic.
 
We  prove here the
following stronger result.

\begin{theorem}\label{retracttheorem}
Let $X$ be a $0$-reduced simplicial set.  Let $\phi$  be the Szczarba map and $\psi$  the retraction map defined above.   

There is a 
natural strong deformation retraction of chain complexes
\[
\xymatrix{\ccobar X\rto<0.8ex>^-\phi&CGX.\lto<0.8ex>^-\psi 
\;\;\ar`r^d^{\qquad\qquad \Phi}[]}
\]
\end{theorem}

Recall that if $A$ and $B$ are chain complexes, $\nabla: A\to B$ and $f:B\to A$ are chain maps, and $h:B\to B$ is linear map of degree $+1$, then 
\[
\xymatrix{A\rto<0.8ex>^-\nabla&B\lto<0.8ex>^-f 
\!\ar`r^d^{\qquad\qquad h}[]}
\]
is a strong deformation retract if $f\nabla=Id_{A}$ and $\dd h+h\dd= \nabla f-Id_{B}$.  Given a strong deformation retract, one can apply the machinery of homological perturbation theory to transfer perturbations of the structure $B$ across to $A$, obtaining a new strong deformation retract \cite {gugenheim-lambe-stasheff:i}, \cite{gugenheim-lambe-stasheff:ii}, \cite{lambe-stasheff}.

\begin{proof} According to Proposition \ref{psiphi}, we need only to prove that there is a natural chain homotopy from the composite $\phi\psi $ to the identity map on $CGX$.  The proof, an acyclic models argument, proceeds by induction on degree.

The base step of the induction is trivial, by Proposition \ref{degreezero}.  We can simply set $\Phi_{0}=0: C_{0}GX\to C_{1}GX$ for all $0$-reduced simplicial sets $X$.

Suppose now that $\Phi_{k}:C_{k}GX \to C_{k+1}GX$ has been defined for all $0\leq k<n$ and for all $0$-reduced simplicial sets $X$ so that 
\begin{itemize}
\item [($1.k$)] $\Phi_{k}$ is natural in $X$ for all $k$, and 
\item [($2.k$)] $\dd \Phi_{k}+\Phi_{k-1}\dd =\phi\psi-Id_{CGX}$ for all $k$ and all $X$,
\end{itemize}
where $n\geq 1$.

Let $\overline\Delta [n]$ denote the quotient of the standard simpicial $n$-simplex by its $0$-skeleton.  If $x=(k_{0}\cdots k_{j})$ is {a} $j$-simplex of $\Delta[n]$, let $x\cdot n$ denote the $(j+1)$-simplex $(k_{0}\cdots k_{{j}}\, n)$. Let 
$$h^n_{i}:\big(G\overline\Delta[n]\big )_{i}\to \big(G\overline\Delta[n]\big)_{i+1} $$
denote the group homomorphism specified by $h^n_{i}(\tau x)=\tau (x\cdot n)$ for all $x\in \overline\Delta [n]_{i+1}$.  Let $$\bar h^n:C_{*}G\overline \Delta [n]\to C_{*+1}G\overline \Delta [n]$$ denote the degree $+1$ linear map specified by $\bar h^n_{i}(w)=-h^n_{i}(w)$ for all $w\in (G\overline\Delta[n])_{i}$.

Morace and Prout\'e proved in {\cite{mp}} that for all $i\geq 1$
$$\dd _{i+1}\bar h^n_{i}+ \bar h^n_{i-1}\dd _{i}=Id,$$
i.e., that $\bar h^n$ is a contraction in positive degrees. It follows that $H_{i}G\overline\Delta [n]=0$ for all $i\geq 1$.

{Consider the infinite wedge} 
$$W(n+1)=\bigvee_{m\in \mathbb N}\overline \Delta [n+1].$$
There is a chain homotopy 
$$\tilde h^{n+1}:C_{*}GW(n+1)\to C_{*+1}GW(n+1)$$
that is a contraction in positive degrees and  
that generalizes Morace and Prout\'e's construction.

Let
$$w=\tau \delta_{{m}_{1}}^{\alpha_{1}}\cdot \ldots\cdot \tau \delta_{{m}_{k}}^{\alpha_{k}}\in \big(GW(n+1)\big)_{n},$$
where $\delta_{{m}_{i}}$ denotes the unique nondegenerate
$(n+1)$-simplex in the ${m}_{i}^{\text{th}}$ copy of
$\overline\Delta [n+1]$ in $W(n)$, {and $\alpha_i=\pm1$}. 
Set 
$$\Phi_{n}(w)=\tilde h^n\big(\phi\psi(w)-w-\Phi_{n-1}(\dd w)\big).$$ 
The induction hypothesis implies that $\phi\psi(w)-w-\Phi_{n-1}(\dd w)$ is a cycle and that
\begin{align*}
\dd \Phi _{n}(w)=&-\tilde h^n\dd \big(\phi\psi(w)-w-\Phi_{n-1}(\dd w)\big)+\phi\psi(w)-w-\Phi_{n-1}(\dd w)\\
=&\phi\psi(w)-w-\Phi_{n-1}(\dd w).
\end{align*}
 Adding $\Phi_{n-1}(\dd w)$ to both sides of this equation, we obtain
$$\dd \Phi_{n}(w)+\Phi_{n-1}(\dd w)=\phi\psi(w)-w,$$
i.e., ($2.n$) holds for all such $w$.

Let $X$ be any $0$-reduced simplicial set, and let 
$$w=\tau x_{1}^{\alpha_{1}}\cdot \ldots\cdot \tau x_{k}^{\alpha_{k}}$$ 
be any nondegenerate $n$-simplex in $GX$, where $\alpha_{i}=\pm 1$ for all $i$.  Let $\zeta_{i}:\overline\Delta [n+1]\to X$ be the simplicial map representing $x_{i}$.

Let $\gamma: \overline\Delta [n+1]\to X$ denote the simplicial map collapsing everything to the basepoint.  Consider the  morphism of simplicial groups 
$$\Psi_{w}=G(\zeta_{1}\vee\cdots\vee \zeta_{k}\vee \bigvee_{m>k} \gamma): GW(n+1)\to GX.$$
Observe that 
$$\Psi _{w} (\tau \delta_{1}^{\alpha_{1}}\cdot \ldots\cdot \tau \delta_{k}^{\alpha_{k}})=w,$$
where $\delta_{i}$ denotes the unique nondegenerate $n$-simplex in the $i^{\text{th}}$ copy of $\overline\Delta[n+1]$ in $W(n+1)$.  

Using the map $\Psi_{w}$ constructed above for any generator $w$ of $C_{n}GX$, we define $\Phi_{n}:C_{n}GX\to C_{n+1}GX$ for any $0$-reduced simplicial set $X$ by 
$$\Phi_{n}(w)=C_{n+1}\Psi_{w}\circ \Phi_{n}(\tau \delta_{1}^{\alpha_{1}}\cdot \ldots\cdot \tau \delta_{k}^{\alpha_{k}}).$$
Note that if $X=W(n+1)$ and $w=\tau \delta_{j_{1}}^{\alpha_{1}}\cdot \ldots\cdot \tau \delta_{j_{k}}^{\alpha_{k}}$, then $$\Psi _{w}:GW(n+1)\to GW(n+1)$$ is a homomorphism of simplicial groups given simply by permuting generators.  It follows from the  construction of the chain homotopy $\bar h^{n+1}$ and therefore of the chain homotopy $\tilde h^{n+1}$ that $\tilde h^{n+1}$ is natural with respect to homomorphisms that simply permute generators, so that
$$C_{n+1}\Psi_{w}\circ \tilde h^n_{n}=\tilde h^n_{n} \circ C_{n}\Psi _{w}.$$
Consequently, $\Phi_{n}(w)$ is indeed well-defined, since $\phi$, $\psi$ and, by the induction hypothesis, $\Phi_{n-1}$ are all natural with respect to simplicial maps.

Moreover,
\begin{align*}
\dd \Phi_{n}(w)&=C_{n+1}\Psi_{w}\circ \dd \Phi_{n}(\tau \delta_{1}^{\alpha_{1}}\cdot \ldots\cdot \tau \delta_{k}^{\alpha_{k}})\\
&=C_{n+1}\Psi_{w}\circ \big( (\phi\psi -Id_{CGW(n)} -\Phi_{n-1}\dd )(\tau \delta_{1}^{\alpha_{1}}\cdot \ldots\cdot \tau \delta_{k}^{\alpha_{k}})\big)\\
&\overset {(\star)}{=}(\phi\psi -Id_{CGX} -\Phi_{n-1}\dd )\circ C_{n}\Psi_{w}(\tau \delta_{1}^{\alpha_{1}}\cdot \ldots\cdot \tau \delta_{k}^{\alpha_{k}})\\
&=(\phi\psi -Id_{CGX} -\Phi_{n-1}\dd )(w),
\end{align*}
where the equality $(\star)$ follows from naturality of $\phi$, $\psi$ and $\Phi_{n-1}$.  In other words,
$$\dd \Phi_{n}+\Phi_{n-1}\dd =\phi\psi- Id_{CGX},$$
for all $X$, i.e., condition $(2.n)$ holds.  

To conclude, observe that condition $(1.n)$ holds as well, since for all simplicial maps $g:X\to Y$ between $0$-reduced spaces and all $w\in GX$, 
$$Gg\circ \Psi_{w}=\Psi_{Gg(w)}:GW(n+1)\to GY.$$
\end{proof}

\begin{remark}  It is in order to be able to apply the chain homotopy of Morace and Prout\'e that we work with $0$-reduced simplicial sets.  There is no such chain homotopy in the $1$-reduced case, so it seems we are obliged to prove the existence of the strong deformation retract in the $0$-reduced case in order to conclude that it exists in the $1$-reduced case as well.
\end{remark}

\begin{remark}  As defined in the proof above, $\Phi_{n}$ is almost certainly not a derivation homotopy, since, as easy computations show, $\tilde h$ is not a derivation homotopy.
\end{remark}

\begin{remark}  Morace and Prout\'e showed that $\bar h^n_{i+1} \circ \bar h^n_{i}=0$ for all $i$ and $n$, from which it follows that $\tilde h_{i+1}^n\circ \tilde h_{i}^n=0$ as well and therefore that 
$$\tilde h_{n+1}^{n+1}\circ \Phi_{n}(\tau \delta_{j_{1}}^{\alpha_{1}}\cdot \ldots\cdot \tau \delta_{j_{k}}^{\alpha_{k}})=0$$
for all $j_{1},...,j_{k}$ and $\alpha_{1},...,\alpha _{k}$.
\end{remark}

\begin{remark} The results in this paper
  generalise from chain complexes to crossed complexes. There is a
  crossed cobar construction $\Omega\pi X$ on the
  fundamental crossed complex $\pi X$, see~\cite{twisted}, and
  we may define a `crossed' Szczarba map of crossed chain algebras $\phi:\Omega \pi
  X\to \pi GX$ that forms part of a deformation retraction
\[
\xymatrix{\Omega\pi X\rto<0.8ex>^-\phi&\pi GX.\lto<0.8ex>^-\psi 
\;\;\ar`r^d^{\qquad\qquad \Phi}[]}
\]
The classical argument that
$\phi$ is a weak equivalence, using 
\adraft{\eqref{sesmap}}, 
does not go through
in this slightly non-abelian 
situation,
since it seems there is no good notion of twisted tensor product of
crossed complexes.
\end{remark}

\bigskip
 
\bigskip

\begin{center}
\textsc{Acknowledgements}
\end{center}

The second author would like to warmly thank the Centre de Recerca
Matem\`atica University 
for the excellent working
conditions provided in the final stages of this paper.
A.T. is partially supported by MEC/FEDER grant MTM2007-63277.

\bigskip
 
\bigskip

\bibliographystyle{amsplain}
\bibliography{szcz}

\bigskip
 
\bigskip

\end{document}